\documentclass[11pt, a4paper]{article}

\usepackage{stmaryrd}

\usepackage[cp1250]{inputenc}

\usepackage{amsmath}
\usepackage{amsthm}
\usepackage{amssymb}
\usepackage{amsfonts}
\usepackage{mathrsfs}
\usepackage{enumerate}

\newtheorem{theorem}{Theorem}
\newtheorem{proposition}[theorem]{Proposition}
\newtheorem{lemma}[theorem]{Lemma}
\newtheorem{corollary}[theorem]{Corollary}

\newtheorem*{switchinglemma}{Lemma \ref{lem:switching}}

\theoremstyle{definition}
\newtheorem{definition}[theorem]{Definition}
\newtheorem*{conjecture}{Conjecture}

\theoremstyle{remark}

\newtheorem*{open}{Open problem}

\newcommand{\imp}{\Rightarrow}

\newcommand{\tuple}[1]{\langle #1 \rangle}

\newcommand{\bbn}{\mathbb{N}}

\newcommand{\kk}{{\mathcal K}}
\newcommand{\n}{{\mathcal N}}
\newcommand{\m}{{\mathcal M}}

\renewcommand{\iff}{\Leftrightarrow}

\newcommand{\ba}{\bar a}

\newcommand{\bx}{\bar x}

\newcommand{\pa}{\mathrm{PA}}
\newcommand{\bba}{\mathrm{BA}}

\newcommand{\pv}{\mathrm{PV}}

\newcommand{\lin}{\mathrm{lin}}

\DeclareMathOperator{\Log}{{Log}}
\newcommand{\ph}{\mathrm{PH}}
\newcommand{\linh}{\mathrm{LinH}}
\newcommand{\eh}{\mathrm{EH}}

\newcommand{\Exp}{\mathrm{Exp}}

\newcommand{\es}[2]{\exists #1 \! < \! #2 \,}   
\newcommand{\as}[2]{\forall #1 \! < \! #2 \,}
\newcommand{\eb}[2]{\exists #1 \! \le \! #2 \,}
\newcommand{\ab}[2]{\forall #1 \! \le \! #2 \,}

\newcommand{\ido}{\mathrm{I}\Delta_0}
\newcommand{\bsi}{\mathrm{B}\Sigma_1}

\title{End-extensions of models of weak arithmetic \\ from complexity-theoretic containments }
\author{Leszek Aleksander Ko{\l}odziejczyk\footnote{Institute of Mathematics, 
University of Warsaw, Banacha 2, 02-097 Warszawa, Poland, \texttt{lak@mimuw.edu.pl}.
Partially supported by Polish National Science Centre grant no. 2013/09/B/ST1/04390.}}

\begin{document}

\maketitle

\begin{abstract} We prove that if the linear-time and polynomial-time hierarchies coincide, then every model of $\Pi_1(\bbn) + \neg \Omega_1$ has a proper end-extension to a model of $\Pi_1(\bbn)$, and so $\Pi_1(\bbn) + \neg \Omega_1 \vdash \bsi$. Under an even stronger complexity-theoretic assumption which nevertheless seems hard to disprove using present-day methods, $\Pi_1(\bbn) + \neg \Exp \vdash \bsi$. Both assumptions can be modified to versions which make it possible to replace $\Pi_1(\bbn)$ by $\ido$ as the base theory. 

We also show that any proof that $\ido + \neg \Exp$ does not prove a given finite fragment of $\bsi$ has to be ``non-relativizing'', in the sense that it will not work in the presence of an arbitrary oracle. 
\end{abstract}

The work presented below aims at a better understanding of the following notoriously hard open problem about 
moderately weak theories of arithmetic:
\begin{equation}\label{eqn:collection}
\tag{$\star$} \textrm{Does }\mathrm{I}\Delta_0 + \neg \Exp \textrm{ prove }\mathrm{B}\Sigma_1?
\end{equation}
Here $\ido$ is induction for bounded formulas in the language of ordered rings, 
$\Exp$ is the axiom $\forall x\, \exists y\, (y =2^x)$ with $y = 2^x$ expressed by an appropriate $\Delta_0$ formula,
and $\bsi$ is the $\Sigma_1$ collection scheme,
\[\ab{x}{u} \exists y\, \varphi(x,y) \imp \exists w\, \ab{x}{u} \eb{y}{w} \varphi(x,y),\] 
where $\varphi$ is $\Delta_0$ and may contain parameters. 

It is well-known and easy to show that $\ido \not \vdash \bsi$ \cite{parsons:choice},
but all known proofs of this (e.g. \cite{parsons:choice, pk:collection, adamowicz:instances, beklemishev:collection}) 
use objects of at least exponential size, often in the form of a definition of satisfaction for $\Sigma_1$ formulas.

Problem \eqref{eqn:collection} was first explicitly stated in \cite{wp:endextensions}. The expected answer is negative, and there have been a number
of results of the form ``the answer to \eqref{eqn:collection} is negative under some complexity-theoretic assumptions'' \cite{wp:endextensions, ferreira:tally, akp:truthdef}. However, even though the assumptions used to obtain a negative answer have been varied in spirit and in some cases mutually contradictory, an unconditional negative answer remains elusive.

Below, we try to explain part of the difficulty with \eqref{eqn:collection} by pointing out a complexity-theoretic assumption
which implies that the answer to the question ``does $\neg \Exp$ prove $\bsi$?'' is actually \emph{positive}, over $\Pi_1$ truth if the assumption is true and over plain $\ido$ if the assumption is provable in $\ido$ (which has to be stated in a special way). The assumption is not at all likely to be true; what seems much more likely, however, is that disproving it might be beyond the reach of present-day methods of complexity theory. 

Earlier work had already revealed the possibility of complexity-theoretic assumptions implying the provability of fragments of $\bsi$ from $\neg \Exp$, but those fragments have been rather weak. Ferreira \cite{ferreira:endextensions} showed that some extremely restricted instances of $\Sigma_1$ collection are actually unconditionally provable from $\neg \Exp$; his work is related to the correspondence between the relativized linear-time hierarchy and polynomial-size bounded-depth circuits. Recently, 
Cord\'on Franco \emph{et al.} \cite{cfl:provably-delta} came up with a computational condition exactly equivalent to the provability of \emph{parameter-free} $\Sigma_1$ collection in $\Pi_1(\bbn) + \neg\exp$. Their condition concerns the complexity of search problems, and it is not obvious how it relates to the well-studied complexity hierarchies of decision problems.

Our assumption is just one among many statements of the general form ``(a level of) an apparently larger time hierarchy is contained in (a higher level of) an apparently smaller time hierarchy'', which tend to imply the provability of $\bsi$ from an axiom negating the totality of a suitable function $f$. Before discussing $f(x) = 2^x$, we first consider the simpler case of $f(x)$ equal to $\omega_1(x) = 2^{\log^2x}$. 

The connection between assumptions of this sort and $\Sigma_1$ collection involves end-extensions of models. A few specific  constructions used to prove containments between computational complexity classes have also led to end-extension theorems \cite{ferreira:endextensions, zambella:endextensions}. We observe that essentially any containment between levels of typical time hierarchies defined by different time bounds implies the existence of end-extensions for models of corresponding theories. Our focus is on the case where the smaller time bounds are linear, so the corresponding theory is $\ido$. As is well-known, any model of $\ido$ with a proper end-extension has to satisfy $\bsi$.

We also apply the technique of building end-extensions from complexity-theoretic containments to show a result about theories in a language expanded by a new ``oracle'' predicate. Roughly speaking, our theorem says that for every finite fragment of $\bsi$, a proof that $\neg \Exp$ does not imply the fragment has to be ``non-relativizing'', in the sense that it will not work in the presence of an arbitrary oracle. To show this, we go through a lemma in computational complexity: for every $k$, there exists an oracle relative to which the $k$-th level of the exponential-time hierarchy is contained in the linear-time hierarchy.

The paper is structured as follows. We review the necessary definitions and background in Section 1. We prove a very simple version of our ``end-extensions from containments'' result in Section 2 and the version about $\neg \Exp$ in Section 3. The result about relativization is proved across Sections 4 and 5. The short Section 6 contains some concluding remarks.
 
\section{Preliminaries}

We assume that the reader is well-acquainted with weak theories of arithmetic as described for instance in \cite{kra95}. In particular, we assume familiarity with the theory $\ido$, the axioms $\Omega_1$ and $\Exp$, and with Buss' formula classes $\Sigma^b_k$ and theories $T_2$, $T^k_2$, $n \in \bbn$. $L_{\pa}$ is the language of Peano Arithmetic, or in other words the language 
of ordered rings. $L_{\bba}$ is Buss' language for bounded arithmetic. The class of bounded formulas in $L_{\bba}$ is denoted $\Sigma^b_\infty$; the symbol $\Delta_0$ always stands for the bounded formulas of $L_{\pa}$ only. 

Recall from \cite{kpt} that $\pv_{k+1}$ functions are canonically defined symbols for polynomial-time functions with a $\Sigma^p_k$ oracle, where $\Sigma^p_k$ is the $k$-th level of the polynomial-time hierarchy $\ph$.
The set of $\pv_{k+1}$ functions contains a symbol for the characteristic function of every $\Sigma^p_k$ relation and is closed under operations corresponding to composition and bounded recursion on notation. $\pv_{k+1}$ also stands for a theory
in the language with all $\pv_{k+1}$ functions, axiomatized by the defining axioms for the function symbols and induction for all quantifier-free formulas. $\pv_{k+1}$ is an extension by definitions of $T^k_2$ (with some subtleties for $k=0$, cf.\ \cite{jerabek:sharply, bk:sharply}), and can be identified with $T^k_2$ for all intents and purposes. 

Most of the computational complexity theory we use is very old-school, and can be found for instance in \cite{bdg95, bdg90}. Recall that a function $f$ is \emph{time-constructible} if there is a Turing machine which on inputs of length $n$ runs for exactly $f(n)$ steps. Essentially, a time-constructible function is one which makes sense as a time bound. For $k \ge 1$, a \emph{$\Sigma_k$ machine} is one that begins its computation in an existential (nondeterministic) state, and switches between existential and universal (co-nondeterministic) states at most $k-1$ times. Thus, for instance, $\Sigma^p_k$ is $\Sigma_k$-$\mathrm{TIME}(n^{O(1)})$. The \emph{linear-time hierarchy} $\linh$ is $\bigcup_k \Sigma_k$-$\mathrm{TIME}(O(n))$. A relation on the natural numbers is in $\linh$ exactly if it can be defined by a $\Delta_0$ formula \cite{wrathall}.

A \emph{decision tree} is a finite binary tree $T$ with leaves labelled by 0 or 1, with internal nodes labelled by boolean variables $x_i$, and the two outgoing edges from an internal node labelled by 0 and by 1. We also assume that on any single branch of $T$ each variable appears as a label at most once. The height of $T$ is the length of the longest branch in $T$.
Any boolean assignment $\alpha$ determines a branch through $T$: from a node labelled by $x_i$, the branch takes the edge labelled by $\alpha(x_i)$. A decision tree $T$ with internal nodes labelled by variables from $\{x_1,\ldots,x_n\}$ \emph{computes} (or \emph{decides}) a boolean function $f$ on $\{0,1\}^n$ if for every $\tuple{\alpha(x_1),\ldots,\alpha(x_n)} \in \{0,1\}^n$,
$f(\alpha)$ equals the label of the leaf on the branch determined by $\alpha$.

\section{Simplest case}

\begin{theorem}\label{thm:simple1}
If $\ph = \linh$, then every model $\m \models \Pi_1(\bbn)$ has an end-extension to $\kk \models T_2$.
\end{theorem}

\begin{proof}

Assume that $\ph = \linh$. This implies that for every $\Sigma^b_\infty$ formula $\psi(x)$ there exists a $\Delta_0$ formula $\psi^\lin(x)$ such that $\psi$ and $\psi^\lin$ are equivalent in the standard model of arithmetic.

Let $\m \models \Pi_1(\bbn)$. $\m$ has a $\Delta_0$-elementary extension to a model $\n$ satisfying the true theory of the natural numbers in $L_{\bba}$. Define $\kk$ to be the closure of $\m$ in $\n$ under Skolem functions for $\Sigma^b_\infty$ formulas, where the Skolem function for $\eb{y}{t}\varphi(\bx,y)$ is the ($\Sigma^b_\infty$-definable) function which outputs the smallest $y \le t$ such that $\varphi(\bx,y)$, or $0$ if there is no such $y$.

Clearly, $\kk$ is a $\Sigma^b_\infty$-elementary substructure of $\m$ and a model of $T_2$. 
We need to show that $\kk \supseteq_e \m$.

Let $g$ be the Skolem function for a $\Sigma^b_\infty$ formula  $\eb{y}{t}\varphi(\bx,y)$ 
and  assume that for $\ba, b \in\m$, we have $g^\n(\ba) \leq b$. Then 
$\n \models \eb{y}{b} (y = g(\ba))$, and thus also 
\[ \n \models \eb{y}{b}(y=g(\bx))^\lin[\ba/\bx].\]
By $\Delta_0$-elementarity, the  same formula holds in $\m$, so there is $c \in \m$ such that 
$\m \models (y=g(\bx))^\lin[\ba/\bx, c/y]$. Thus $\n \models (y=g(\bx))^\lin[\ba/\bx, c/y]$
and $\n \models c = g(\ba)$, which means that  $g^\n(\ba) \in \m$. 
\end{proof}

If we want to end-extend an arbitrary model of $\ido$, rather than just a model of $\Pi_1(\bbn)$, to a structure satisfying $T_2$, we need a stronger assumption.

\begin{theorem}\label{thm:simple2}
Assume that there is a translation $\psi \mapsto \psi^\lin$ of $\Sigma^b_\infty$ into $\Delta_0$ formulas such that 
$T_2 + \{ \forall x \, (\psi(x) \iff \psi^\lin(x)) : \psi \in \Sigma^b_\infty\}$ is $\Pi_1$-conservative over $\ido$. Then every model $\m \models \ido$ has an end-extension to $\kk \models T_2$.
\end{theorem}

\begin{proof}
Consider $\m \models \ido$. Under the assumption of the theorem, $\m$ has a $\Delta_0$-elementary extension to a model $\n$ satisfying $T_2$ in which each $\Sigma^b_\infty$ formula $\psi$ is equivalent to $\psi^\lin$. Define $\kk$ to be the closure of $\m$ in $\n$ under Skolem functions for $\Sigma^b_\infty$ formulas and repeat the remainder of the proof of Theorem \ref{thm:simple1}.
\end{proof}

It can be argued that the assumption of Theorem \ref{thm:simple2} is a reasonable way of making precise 
the statement ``$\ido \vdash \ph = \linh$''. 
In essence, what the assumption says is that a model of $\ido$
can be extended to another structure in which the $\linh$ properties are preserved, 
the $\ph$ properties are reasonably well-behaved (in other words, $T_2$ is satisfied) and the $\ph$ properties are in fact in $\linh$. Our use of this assumption gives rise to the following open problem:
\begin{open}
Is it true that if $T_2$ is $\Pi_1$-conservative over $\ido$, then in fact the assumption of Theorem \ref{thm:simple2} holds?
\end{open}

Note that one other possible formulation of ``$\ido \vdash \ph \subseteq \linh$'', namely that provably in $\ido$, if $x \# x, x \# x \# x, \ldots$ exist, then each $\Sigma^b_\infty$ formula $\psi(x)$ has a $\Delta_0$ translation $\psi^\lin(x)$, is unreasonably weak. In fact, it is \emph{weaker} than
$T_2 \vdash \ph = \linh$.

\begin{corollary}\label{cor:simple-coll}
If $\ph = \linh$, then $\Pi_1(\bbn) + \neg \Omega_1 \vdash \bsi$. If the assumption of Theorem \ref{thm:simple2} holds, then $\ido + \neg \Omega_1 \vdash \bsi$.
\end{corollary}
\begin{proof}
If $\m \models \Pi_1(\bbn) + \neg \Omega_1$, then the end-extension given by Theorem \ref{thm:simple1} is of necessity proper, so $\m \models \bsi$. Similarly for $\m \models \ido + \neg \Omega_1$ and Theorem \ref{thm:simple2}.
\end{proof}

\section{The case of $\neg \Exp$}

We now consider the task of formulating an assumption which would imply provability of $\bsi$ from $\neg \Exp$ rather than just from $\neg \Omega_1$. As previously, we first work with $\Pi_1(\bbn)$ as our base theory.

The natural idea would be to replace the polynomial time hierarchy in the statement of Theorem \ref{thm:simple1} by the smaller of the two well-known ``exponential time hierarchies'', which is defined by letting $\Sigma^\mathrm{e}_{k}$ be $\Sigma_k$-$\mathrm{TIME}(2^{O(n)})$ and $\eh$ be $\bigcup_{k \in \bbn}\Sigma^\mathrm{e}_k$. By an argument similar to the proof of Theorem \ref{thm:simple1}, we could show that if $\eh = \linh$, then every $\m \models \Pi_1(\bbn)$ has an end-extension to $\kk \models  T_2$ such that $\m = \Log \kk$, where $\Log \kk$ equals $\{a \in \kk: 2^a \textrm{ exists in } \kk\}$. If $\m \models \neg \Exp$, the extension is proper, so $\m \models \bsi$.

Unfortunately, the assumption $\eh = \linh$ is too strong for our purposes, because it is known to be false. The following proposition is apparently folklore.

\begin{proposition}\label{prop:eh}
$\eh \neq \linh$.
\end{proposition}
\begin{proof}
For each $k \in \bbn$, $\Sigma_k$-$\mathrm{TIME}(O(n)) \subsetneq \Sigma^\mathrm{e}_{k}$ by the nondeterministic time hierarchy theorem. So,
if $\linh = \Sigma_k$-$\mathrm{TIME}(O(n))$, then $\linh \subsetneq \Sigma^\mathrm{e}_{k}$. 

On the other hand, $\mathrm{E} = \mathrm{DTIME}(2^{O(n)})$ contains $\linh$ and has a problem which is complete via linear-time reductions. Thus, if $\linh$ does not collapse to a finite level,  we have $\linh \subsetneq \mathrm{E} \subseteq \Sigma^\mathrm{e}_{1}$. 
\end{proof}

Proposition \ref{prop:eh} means that we have to consider time bounds whose growth rate is ``fractional-exponential'':
\begin{definition} The function $f \colon \bbn \to \bbn$, $f(n) = O(2^n)$, has \emph{fractional-exponential growth rate} if there exists $\ell \in \bbn$ such that \[\underbrace{f \circ f \circ \ldots \circ f}_{\ell \textrm{ times}} \ge 2^n.\]
\end{definition}
Functions of fractional-exponential growth rate appear naturally in some contexts in computational complexity theory \cite{mvw:cocoon}. There are time-constructible functions of arbitrarily slow fractional-exponential growth:

\begin{lemma}
For every $\ell \in \bbn$, there exists a non-decreasing time-constructible function such that  
\[\underbrace{f \circ f \circ \ldots \circ f}_{\ell \textrm{ times}} = o(2^n),\]
but for some $m > \ell$,
\[\underbrace{f \circ f \circ \ldots \circ f}_{m \textrm{ times}} \ge 2^n.\]

\end{lemma}

\begin{proof}
We present aa very straightforward construction which suffices for our purposes. For more subtle results on iterative roots of $\exp$ as a function of a real variable, see e.g. \cite{kneser:reelle},  \cite{szekeres:fractional}.

Define $\exp_k(n)$ by setting $\exp_0(n) := n$, $\exp_{k+1}(n) = 2^{\exp_k(n)}$.

Let $f_2$ be defined as follows. For $n \le 2$, $f_2(n)$ equals $3$. For $n \ge 3$, $f_2(n)$ is the smallest element of the sequence
\[3, 4, 2^3, 2^4, 2^{2^3}, 2^{2^4}, \ldots \]
which is strictly greater than $n$. A routine calculation shows that \mbox{$f_2(n) = o(2^n)$}, but $f_2 \circ f_2 \circ f_2 (n) > 2^n$.

For general $k \ge 1$, define $f_k$ as follows. For $n \le \exp_{k-1}(2)$, $f_k(n)$ equals $\exp_{k-1}(2) + 1$. For $n \ge \exp_{k-1}(2) + 1$, $f_k(n)$ is the smallest element of the sequence
\begin{multline*}
\exp_{k-1}(2) +1, \exp_{k-1}(2) +2, \ldots, \exp_k(2), \\ \exp(\exp_{k-1}(2) +1), \exp(\exp_{k-1}(2) +2), \ldots, \exp_{k+1}(2),\\
\exp_2(\exp_{k-1}(2) +1), \exp_2(\exp_{k-1}(2) +1), \ldots
\end{multline*}
which is strictly greater than $n$. A calculation shows that 
\[\underbrace{f_k \circ f_k \circ \ldots \circ f_k}_{\ell \textrm{ times}} = o(2^n) , \qquad \underbrace{f_k \circ f_k \circ \ldots \circ f_k}_{\ell+2 \textrm{ times}} > 2^n \] 
for $\ell := \exp_k(2) - \exp_{k-1}(2) - 1$.

It is not immediately obvious if the functions $f_k$ are time-constructible. 
However, let $\tilde f_k(n)$ be the time needed to compute $f_k$ on argument $n$ with both input and output in unary (on a fixed machine). 
By definition, $\tilde f_k$ is time-constructible. Moreover, it is clear from the definition of $f_k$  that for a reasonable choice of a machine computing $f_k$, $\tilde f_k$ will be non-decreasing and satisfy:
\[ f_k \le \tilde f_k = O(f_k),\]
which also implies
\[\underbrace{\tilde f_k \circ \tilde f_k \circ \ldots \tilde \circ f_k}_{\ell-1 \textrm{ times}} = o(2^n),\]
for $\ell$ as above.
\end{proof}

For every time-constructible $f$ of fractional-exponential growth, it is still the case $\linh$ is strictly contained in $\bigcup_k \Sigma_k$-$\mathrm{TIME}(f)$, the quantifier hierarchy determined by $f$ (cf.\ \cite[Theorem V.2.21]{hp93}). This is because a straightforward padding argument shows that $\bigcup_k \Sigma_k$-$\mathrm{TIME}(f) \subseteq \linh$ implies $\bigcup_k \Sigma_k$-$\mathrm{TIME}(f \circ f) \subseteq \bigcup_k \Sigma_k$-$\mathrm{TIME}(f)$ and so on,
eventually reaching $\eh \subseteq \linh$; but this contradicts Proposition \ref{prop:eh}.

However, the following statement, although highly implausible, appears to be consistent with present knowledge, and disproving it may well be beyond the reach of known methods:
\begin{multline*}\label{eqn:assumption}
\tag{$\spadesuit$}
\textrm{ ``for every $k$ there is some $f$ of fractional-exponential growth} \\
\textrm{such that }\Sigma^{f^{O(1)}}_k \subseteq \linh\textrm{''.}
\end{multline*}
By the discussion above, $f$, or more precisely the number $\ell$ of iterations of $f$ needed to reach $\exp$, would have to depend on $k$. A very slight modification of statement (\ref{eqn:assumption}) is the assumption we use to prove that 
$\Pi_1(\bbn) + \neg \Exp \vdash \bsi$.

\begin{theorem}\label{thm:full1}
If there exists a non-decreasing time-constructible function $f$ of fractional-exponential growth such that $\mathrm{DTIME}(f^{O(1)})^{\Sigma^p_k} \subseteq \linh$, then every model $\m \models \Pi_1(\bbn) + \neg \Exp$ has a proper end-extension to $\kk \models T^k_2$.
\end{theorem}

\begin{proof}

Assume that $f$ is non-decreasing, has fractional-exponential growth, and $\mathrm{DTIME}(f^{O(1)})^{\Sigma^p_k}$ is contained in $\linh$. 
Let $h$ be the time-constructible function defined by $h(x) = 2^{f(|x|)}$. Note that for an input $x$ with $|x|=n$, the value $h(x)$ (written in binary) is computable in deterministic time $f(n)$. Because of this and the monotonicity of $f$, for every $\pv_{k+1}$ formula $\psi(x,y)$, the formula $\psi(x,h(y))$ defines a property in $\mathrm{DTIME}(f^{O(1)})^{\Sigma^p_k}$. Therefore, by our assumption, there exists a $\Delta_0$ formula $\psi^{f\to\lin}(x,y)$ such that $\psi(x,h(y))$ and $\psi^{f\to\lin}(x,y)$ are equivalent in $\bbn$.

Let $\m \models \Pi_1(\bbn) + \neg \Exp$. Since the totality of $f$ implies the totality of the exponential function, $\Log \m$ is not closed under $f$, so there must be $d \in \m$ such that $h(d) \notin \m$. As in the proof of Theorem \ref{thm:simple1}, $\m$ has a $\Delta_0$-elementary extension to a model $\n$ satisfying true arithmetic in the language $L_\bba$. Of course, $\n$ contains $h(d)$. Define $\kk$ to be the closure of $\m \cup \{h(d)\}$ in $\n$ under $\pv_{k+1}$ functions. By construction, $\kk$ is a proper extension of $\m$, a $\pv_{k+1}$-elementary substructure of $\m$ and a model of $T^k_2$; we only need to show that $\kk \supseteq_e \m$.

Let $g$ be a $\pv_{k+1}$ function and  assume that for $\ba, b \in\m$, we have $g^\n(\ba,h(d)) \leq b$. Then 
$\n \models \es{z}{b} (z = g(\ba,h(d)))$, and thus also 
\[ \n \models \es{z}{b}(z=g(\bx,y))^{f\to\lin}[\ba/\bx,d/y].\]
By $\Delta_0$-elementarity, the  same formula holds in $\m$, so there is $c \in \m$ such that 
\[\m \models (z=g(\bx,y))^{f\to\lin}[\ba/\bx, d/y,c/z].\] Thus also \[\n \models (z=g(\bx,y))^{f\to\lin}[\ba/\bx,  d/y,c/z],\]
which means that $\n \models c = g(\ba,h(d))$. In other words, $g^\n(\ba,h(d)) \in \m$. 
\end{proof}

As in the context of $\neg\Omega_1$, end-extending arbitrary models of $\ido+\neg\exp$ rather than just models of $\Pi_1(\bbn)+\neg\exp$ requires a stronger assumption. The assumption in Theorem \ref{thm:full2} below can be seen as a formalization of the statement ``$\ido \vdash \mathrm{DTIME}(f^{O(1)})^{\Sigma^p_k} \subseteq \linh$''.

The assumption  in Theorem \ref{thm:full2} is more complicated than the one in Theorem \ref{thm:simple2}, and it mentions an additional constant. Another way of formulating the assumption would involve a two-sorted language. The main reason for the complications is that if we introduce a symbol for a function of fractional-exponential growth and allow terms involving nesting of the new function symbol, we will get a theory which proves the totality of $\exp$ and thus has no chance of being $\Pi_1$-conservative over $\ido$. 

\begin{theorem}\label{thm:full2}\, 
Assume that there is a time-constructible function $f$ which is provably non-decreasing in $\ido$ and satisfies the following:
\begin{itemize}
\item[(i)] for some $\ell \in \bbn$, 
\[\ido \vdash \forall x\, \forall y\,(y = \underbrace{f\circ f \circ \ldots \circ f}_{\ell \textrm{ times}}(x) \imp \eb{z}{y} (z = 2^x)).\]
\item[(ii)] for $d$ a new individual constant, there is a translation $\psi \mapsto \psi^{f\to\lin}$ of $\pv_{k+1}$ into $\Delta_0$ formulas such that \[T^k_2 + \exists y\, (y = 2^{f(|d|)}) + \{ \forall x \, (\psi(x,2^{f(|d|)}) \iff \psi^{f\to\lin}(x,d)) : \psi \in \pv_{k+1}\}\] is $\Pi_1(d)$-conservative over $\ido$.
\end{itemize}
 Then every model $\m \models \ido + \neg \Exp$ has a proper end-extension to $\kk \models T^k_2$.
\end{theorem}
\begin{proof}
It suffices to modify the proof of Theorem \ref{thm:full1} in the same way as the proof of Theorem \ref{thm:simple1} was modified to obtain Theorem \ref{thm:simple2}.
\end{proof}

\begin{corollary}\label{cor:full-coll}
If for every $k \in \bbn$ there is a non-decreasing time-constructible $f$ of fractional-exponential growth such that $\Sigma_k$-$\mathrm{TIME}(f^{O(1)}) \subseteq \linh$, then $\Pi_1(\bbn) + \neg \Exp \vdash \bsi$. If the assumption of Theorem \ref{thm:full2} holds for every $k \in \bbn$, then $\ido + \neg \Exp \vdash \bsi$.
\end{corollary}
\begin{proof}
To prove the first  statement, notice that $\mathrm{DTIME}(f^{O(1)})^{\Sigma^p_k}$ is contained in $\Sigma_k$-$\mathrm{TIME}(f^{O(1)})$.
Hence, if $\Sigma_k$-$\mathrm{TIME}(f^{O(1)}) \subseteq \linh$, then by Theorem \ref{thm:full1} every model $\m \models \Pi_1(\bbn) + \neg \Exp$ has an proper end-extension to a model of $T^k_2$. If this happens for every $k$, then $\m$ has a proper end-extension
to a model of $T^k_2$ for every $k$, which, by the usual argument, implies $\m \models \bsi$.

The second statement is proved in a similar way, using Theorem \ref{thm:full2} instead of Theorem \ref{thm:full1}.
\end{proof}

\section{A relativized result}

In this section, we take up the question whether a proof of $\ido + \neg\exp \not\vdash \bsi$ could relativize to an arbitrary oracle. The ideal result here would be to construct an oracle relative to which statement \eqref{eqn:assumption} holds, and use that to show that a proof of $\ido + \neg\exp \not\vdash \bsi$ would have to be ``non-relativizing''. Our theorem below is weaker than that; nevertheless, it is strong enough to highlight a major difference with the case without $\neg \Exp$.

In attempting to discuss relativizations of $\ido + \neg\exp \not\vdash \bsi$, we encounter the problem that it is less obvious how to relativize an unprovability statement than a provability statement. If $\alpha$ is a new predicate symbol (representing a potentially arbitrary oracle) and two theories $T$ and $S$ are such that it makes sense to speak of their relativized versions $T(\alpha), S(\alpha)$, then ``$T \vdash S$ does not relativize''  is simply ``$T(\alpha) \not\vdash S(\alpha)$''. However, what does ``$T \not \vdash S$ does not relativize'' mean? ``$T(\alpha) \vdash S(\alpha)$'' is far too strong, and in particular implies $T \vdash S$. On the other hand ``there are additional axioms about $\alpha$ such that $T(\alpha)$ extended by the new axioms proves $S(\alpha)$'' is obviously too weak, because the new axioms could include $S(\alpha)$ itself.  

In the specific case of $\ido + \neg \Exp \not\vdash \bsi$, we are in the fortunate situation that the right-hand side of the $\not \vdash$ symbol has higher quantifier complexity than the left-hand side. This lets us explicate ``$\ido + \neg \Exp \not\vdash \bsi$ does not relativize'' as ``there are low-complexity axioms about $\alpha$ such that \mbox{$\ido(\alpha) + \neg \Exp$} is consistent and proves $\bsi(\alpha)$'', where the ``low-complexity axioms'' can for instance say that some local conditions on $\alpha$ are satisfied at every input; in other words, they can be $\Pi_1(\alpha)$ sentences.

We conjecture the following:

\begin{conjecture}
Let $\alpha$ be a new unary relation symbol. There exists a consistent recursively axiomatized $\Pi_1(\alpha)$ theory
$T(\alpha) \supseteq \ido(\alpha)$ such that $T(\alpha) + \neg \Exp \vdash \bsi(\alpha)$.
\end{conjecture}

Currently,  we are only able to obtain a weaker theorem:

\begin{theorem}\label{thm:relativization}
Let $\alpha$ be a new unary relation symbol. Then for every finite fragment $\mathrm{B}(\alpha)$ of $\bsi(\alpha)$ there exists a consistent recursively axiomatized $\Pi_1(\alpha)$ theory
$T_B(\alpha) \supseteq \ido(\alpha)$ such that $T_B(\alpha) + \neg \Exp \vdash \mathrm{B}(\alpha)$.
\end{theorem}

Already this result contrasts with the situation without the $\neg \Exp$ axiom.

\begin{proposition}
Let $\alpha$ be a new unary relation symbol. There is a finite fragment $\mathrm{B}(\alpha)$ of $\bsi(\alpha)$ such that for every set $S(\alpha)$ of $\Pi_2(\alpha)$ sentences which is consistent with $\ido(\alpha) + \Exp$ (in particular, for every consistent extension of $\ido(\alpha)$ by a set of $\Pi_1(\alpha)$ sentences), $S(\alpha) \not \vdash \mathrm{B}(\alpha)$.
\end{proposition}
\begin{proof}
Take a nonstandard model $\m \models {\ido(\alpha)} + {\Exp} + {S}$, a nonstandard element $a \in \m$, and let $\kk$ consist of those elements of $\m$ which are $\Sigma_1(\alpha)$-definable with parameter $a$. By the usual Paris-Kirby argument,
$\kk \not \models \bsi(\alpha)$, but $\kk \preccurlyeq_{\Sigma_1(\alpha)} \m$, so in particular $\kk \models S$. This completes the proof, since $\bsi(\alpha)$ is finitely axiomatizable over $\ido(\alpha) + \exp$. 
\end{proof}

Our proof of Theorem \ref{thm:relativization} is based on the following computational lemma.

\begin{lemma}\label{lem:oracle}
For every $k$, there is an oracle $\alpha$ such that $\Sigma^e_k(\alpha) \subseteq \linh(\alpha)$.
\end{lemma}

We first show how the theorem follows from the lemma and then devote the remainder of the section to a proof of the lemma. 

Assume that Lemma \ref{lem:oracle} holds and let $\mathrm{B}(\alpha)$ be a finite fragment of $\bsi(\alpha)$. Let $k$ be such that every model of $\ido(\alpha)$ with a proper end-extension to a model of $T^k_2(\alpha)$ satsifies $\mathrm{B}(\alpha)$. 
Let $T_B(\alpha)$ consist of the $\Pi_1(\alpha)$ consequences of the theory 
$\ido(\alpha) + \Exp + \textrm{``}{\Sigma^e_k(\alpha) \subseteq \linh(\alpha)}\textrm{''}$, where the containment is expressed as the equivalence of some fixed formula defining a property which is $\Sigma^e_k(\alpha)$-complete via linear-time reductions with a $\Delta_0(\alpha)$ formula. By Lemma \ref{lem:oracle}, $T_B(\alpha)$ is consistent, and since it has a recursively enumerable axiomatization it is also recursively axiomatized.

By relativizing the proof of Theorem \ref{thm:full2}, which presents no significant obstacles, we show that every model of $T_B(\alpha) + \neg \Exp$ has a proper end-extension to a model of $T^k_2(\alpha)$. Thus, $T_B(\alpha) + \neg \Exp \vdash \mathrm{B}(\alpha)$, which completes the proof of Theorem \ref{thm:relativization} from Lemma \ref{lem:oracle}.

\begin{proof}[Proof of Lemma \ref{lem:oracle}] For $k = 0, 1$, the result is known \cite{dekhtyar:relativization, heller:relativized}. For $k \ge 2$, our strategy will be to apply a variant of ``H{\aa}stad's second switching lemma'' (Lemma 6.3 in \cite{hastad:thesis}) $k$ times. As a result, we will partially determine the oracle $\alpha$ in such a way that a given $\Sigma^e_k(\alpha)$ property will behave roughly like a deterministic exponential-time property (more precisely, an exponential-height decision tree) relative to the undetermined part of $\alpha$, whereas a well-chosen $\linh$ property will behave like a nondeterministic linear-time property. At that point, it will remain to repeat the construction of an oracle relative to which $\mathrm{E} \subseteq \mathrm{NLIN}$ as presented in \cite{heller:relativized}.

We describe in detail the case of $k =2$. We choose some reasonable coding of tuples and  assume that all inputs to $\alpha$ are quadruples of the form $\tuple{a, y_1, y_2, y_3}$, where 
$y_1 , y_2 < 2^{4|a|}$, $y_3 < 2^{2|a|}$. We also use the convention that whenever $a$ appears in the proof, $N$ stands for $2^{|a|}$, and whenever $b$ appears in the proof, $M$ stands for $2^{|b|}$. Note that a computation on input $a$ is linear-time if it takes time linear in $|a|$ and exponential-time (as understood in the definition of $\eh$) if it takes time polynomial in $N$.

There is a problem in $\Sigma_2$-$\mathrm{TIME}(2^{(1/110)n})(\alpha)$ which is complete for $\Sigma^e_2(\alpha)$ via linear-time reductions. Therefore, it suffices to construct $\alpha$ so that the property $\varphi^\alpha(a)$:
\[ \es{x_1}{2^{N^{1/110}}}\as{x_2}{2^{N^{1/110}}}\psi^\alpha(a,x_1,x_2),\]
where $\psi^\alpha$ makes at most $N^{1/110}$  queries to $\alpha$, is for sufficiently large $a$ equivalent to
\[\as{y_1}{N^4}\es{y_2}{N^4}\as{y_3}{N^2}\alpha(a,y_1,y_2,y_3).\]

We think of each instance $\alpha(b,y_1,y_2,y_3)$ as a propositional variable, which is set to $1$ exactly if $\tuple{b,y_1,y_2,y_3} \in \alpha$. For each fixed $a$, $\varphi^\alpha(a)$ can be viewed as a propositional formula in the variables $\alpha(b,y_1,y_2,y_3)$. The formula is a disjunction of $2^{N^{1/110}}$ CNF's, each of which is a conjunction of disjunctions of at most ${N^{1/110}}$ atoms (an ${N^{1/110}}$-CNF). 

We will refer to a partial assignment to the propositional variables as a \emph{restriction}. We choose a random restriction $\rho$ by the following process:
\begin{itemize}
\item divide all possible inputs to $\alpha$ into blocks of the form $\tuple{b, y_1, y_2, \cdot}$, where each block consists of $M^2$ elements corresponding to the possible values of $y_3$,
\item for each input $\tuple{b, y_1, y_2, y_3}$ independently, set $\alpha(b, y_1, y_2, y_3)$ to $1$ 
with probability $1 - \frac{1}{M}$ and ``set it to $*$'' (in other words, leave it unassigned, or ``starred'')
with probability $\frac{1}{M}$,
\item for each block $\tuple{b, y_1, y_2, \cdot}$ in which not all the values have been set to $1$, make it a ``$0$-block'' (setting 
$\tuple{b, y_1, y_2, y_3}$ to $0$ for all $y_3$ for which this was previously set to $*$) with probability $1 - \frac{1}{M}$ and make it a ``$*$-block'' (keeping the variables currently set to $*$ unassigned) with probability $\frac{1}{M}$.
\end{itemize} 

The restriction $\rho$ is then extended to $g(\rho)$ which additionally sets to $1$ all starred variables in a $*$-block except the one with the smallest value of $y_3$. 

The following variant of H{\aa}stad's switching lemma will guarantee the existence of a restriction $\rho$ with some desirable properties. The proof of the lemma is deferred until Section \ref{sec:switching}.

\begin{lemma}\label{lem:switching}
Let $0 < \delta, \epsilon \le 1$ such that $12 \delta < \epsilon$. Assume $N$ is large enough and let $\psi$ be an $N^\delta$-DNF. Then with probability at least $1 - 2^{-N^\epsilon}$, $\psi|_{g(\rho)}$ can be decided by a decision tree of height less than
$N^\epsilon$. 
\end{lemma}

The probability that a random $\rho$ sets all $M^2$ variables in a given block $\tuple{b, y_1, y_2, \cdot}$ to $1$ is $(1 - \frac{1}{M})^{M^2}$, thus exponentially small in $M$. Similarly, by Chernoff bounds, the probability that fewer than $M^2$ of the $M^4$ blocks $\tuple{b, y_1, y_2, \cdot}$ for a given pair $\tuple{b,y_1}$ become $*$-blocks under $\rho$ is exponentially small in $M$. For any given $M$, there are only $O(M^9)$ triples $\tuple{b, y_1, y_2}$ and pairs $\tuple{b,y_1}$ with $2^{|b|} = M$;
and for any given $N$, there are no more than $N$ inputs $a$ with $2^{|a|} = N$.
Therefore, by Lemma \ref{lem:switching} with $\delta = 1/110, \epsilon = 1/9$, there exists a restriction $\rho$ such that:
\begin{itemize}
\item for each sufficiently large $b$ and each $\tuple{b, y_1, y_2}$ there is at least one $y_3$ such that $\alpha(b,y_1,y_2,y_3)|_\rho \neq 1$,
\item for each sufficiently large $b$ and each $\tuple{b, y_1}$ there are at least $M^2$ values of $y_2$ such that $\tuple{b, y_1,y_2}$ is a $*$-block under $\rho$,
\item for each sufficiently large $a$ and each $x_1 < 2^{N^{1/110}}$, the formula 
\[(\as{x_2}{2^{N^{1/110}}}\psi^\alpha(a,x_1,x_2))|_{g(\rho)}\] can be decided by a decision tree of height
$N^{1/9}$.
\end{itemize}

Take some $\rho$ with these properties and assign 0/1 values to variables $\alpha(b, y_1, y_2, y_3)$ according to $g(\rho)$. For each sufficiently large $a$, the formula $\varphi^\alpha(a)|_{g(\rho)}$, that is,
\[\es{x_1}{2^{N^{1/110}}}[\as{x_2}{2^{N^{1/110}}}\psi^\alpha(a,x_1,x_2)]|_{g(\rho)},\]
becomes an $N^{1/9}$-DNF. Extend $g(\rho)$ further to a restriction $h(\rho)$  by turning some $*$-blocks into $0$-blocks so that for each $\tuple{b, y_1}$ with $b$ sufficiently large there are exactly $M^2$ values of $y_2$ for which $\tuple{b, y_1, y_2, \cdot}$ remains a $*$-block. In this way,  
\[[\as{y_1}{N^4}\es{y_2}{N^4}\as{y_3}{N^2}\alpha(a,y_1,y_2,y_3)]|_{h(\rho)}\]
becomes
\[\bigwedge_{y_1 < N^4}\bigvee_{y_2 \in S(a,y_1)} \alpha(a,y_1,y_2,y_3(a,y_1,y_2)),\]
where the set $S(a,y_1)$ consists of exactly $N^2$ values of $y_2$, and for each such value, $y_3(a,y_1,y_2)$ is the unique $y_3$ for which  $\alpha(a,y_1,y_2,y_3)$ has not been set to $1$.  

We think of inputs to $\alpha|_{h(\rho)}$ as triples $\tuple{b,y_1,y_2}$ (since $\alpha(b,y_1,y_2,y_3)$ can be $0$ for at most one value of $y_3$). We divide the tuples into blocks of the form $\tuple{b,y_1,\cdot}$, so that a sufficiently large $b$ corresponds to exactly $M^4$ blocks containing exactly $M^2$ variables each.  We define a random distribution on restrictions $\hat \rho$ similarly as before, but with $\alpha$ replaced by $\alpha|_{h(\rho)}$, the old blocks replaced by the new blocks, and with the roles of $0$ and $1$ interchanged. In this setting, we can prove an analogue of Lemma \ref{lem:switching} in which the requirement $12\delta < \epsilon$ can be weakened to $8 \delta < \epsilon$ (the reason for the last change is that for every $b$ there are now $M^6$ variables instead of $M^{10}$). Applying this lemma with $\delta = 1/9$ and $\epsilon = 1$, we get a restriction $\hat\rho$ such that for each sufficiently large $a$ the formula
$\varphi^\alpha(a)|_{h(\rho)g(\hat\rho)}$ can be decided by a decision tree of height $N$, and extend $g(\hat\rho)$ to $h(\hat\rho)$ so that 
\[[\as{y_1}{N^4}\es{y_2}{N^4}\as{y_3}{N^2}\alpha(a,y_1,y_2,y_3)]|_{h(\rho)h(\hat \rho)}\]
becomes
\[\bigwedge_{y_1 \in S(a)} \alpha(a,y_1, y_2(a,y_1), y_3(a,y_1)),\]
where $S(a)$ is a set of size exactly $N^2$.

We then complete the construction by copying the proof of Theorem 6 of \cite{heller:relativized}. For $a = 0, 1, \dots$, we first find the value of $\varphi^\alpha(a)$ by answering $1$ to each query to $\alpha$ which has not been
already determined by $h(\rho)h(\tilde\rho)$ and the procedure for lower values of $a$. Assuming $a$ is large enough, this involves answering at most $N$ queries, because $\varphi^\alpha(a)|_{h(\rho)h(\hat\rho)}$  can be decided by a height-$N$ decision tree. Afterwards, of the $N^2$ variables 
$\alpha(a,y_1, y_2(a,y_1), y_3(a,y_1))$ remaining unset after $h(\rho)h(\tilde\rho)$, at most $aN + O(1) \le N\log N + O(1)$ have been set (to $1$) during the procedure for $0, 1, \ldots, a-1$, so at least one remains unset. If $\varphi^\alpha(a)$ is false, we choose one such variable and set it to $0$, setting all other such variables to $1$. If $\varphi^\alpha(a)$ is true, we set all such variables to $1$.
In this way, we guarantee that $\psi^\alpha(a)$ is equivalent to 
\[\as{y_1}{N^4}\es{y_2}{N^4}\as{y_3}{N^2}\alpha(a,y_1,y_2,y_3)\]
for every large enough value of $a$.

This completes the proof of Lemma \ref{lem:oracle} for the case $k=2$. The case for general $k$ is very similar: for a sufficiently large constant $\ell$, we take a property $\varphi^\alpha(a)$ which is complete for $\Sigma^e_k$ via linear-time reductions and can be written as 
\[ \es{x_1}{2^{N^{1/\ell}}}\as{x_2}{2^{N^{1/\ell}}}\ldots \mathrm{Q}x_k \! < \! 2^{N^{1/\ell}} \psi^\alpha(a,x_1,x_2, \ldots, x_\ell),\]
where $\psi$ can be decided in deterministic time $N^{1/\ell}$. The aim is to construct $\alpha$ so that this property becomes equivalent to 
\[\as{y_1}{N^4}\es{y_2}{N^4}\ldots\mathrm{Q}y_k \! < \! N^4 \mathrm{\overline Q}y_{k+1} \! < \! N^2\,\alpha(a,y_1,y_2,\ldots,y_k,y_{k+1})\]
for large enough $a$. To achieve this, we assume that inputs to $\alpha$ have the form $\tuple{b,y_1,\ldots,y_k, y_{k+1}}$ for
$y_1, \ldots, y_k < M^4$ and $y_{k+1} < M^2$, divide the inputs into blocks of the form $\tuple{b,y_1,\ldots,y_k, \cdot}$, and iterate restricting $\alpha$ using analogues of Lemma \ref{lem:switching}, now $k$ times instead of $2$ times. For the $i$-th application of the switching lemma, a block contains inputs corresponding to one fixed tuple $\tuple{b,y_1, \ldots, y_{k+1-i}}$ and various possible values of $y_{k+2-i}$, and the requirement on $\delta$ and $\epsilon$ is $4(k+2-i)\delta < \epsilon$. After $k$ rounds of restricting $\alpha$, $\varphi^\alpha(a)$ becomes computable by a decision tree of height $N$ and we can complete the argument as in the case $k=2$. \end{proof}

\section{The switching lemma}\label{sec:switching}

We repeat the statement of Lemma \ref{lem:switching}. For definitions of $N, \rho$ etc.\ refer to the previous section.

\begin{switchinglemma}
Let $0 < \delta, \epsilon \le 1$ such that $12 \delta < \epsilon$. Assume $N$ is large enough and let $\psi$ be an $N^\delta$-DNF. Then with probability at least $1 - 2^{-N^\epsilon}$, $\psi|_{g(\rho)}$ can be decided by a decision tree of height less than
$N^\epsilon$. 
\end{switchinglemma}

\begin{proof}
The argument is an essentially standard Razborov-style proof of a switching lemma; our presentation is very strongly inspired by Thapen's note \cite{thapen:switching}. One difference in comparison with the usual setting for switching lemmas is that we simultaneously consider blocks of different sizes, and the probability of assigning a value to a variable/block depends on the size of the block. We 
deal with this by dividing blocks into ``large'' and ``small''; our decision tree will deal with the ``small'' blocks by brute-force search, and with the ``large'' ones in the usual way.

Note that  $\psi$ contains only variables from finitely many blocks. For this reason, we may think of our restrictions $\rho$ as defined on finitely many blocks, so that individual restrictions have non-zero probabilities.

Consider the tree $T(\psi, \rho)$ defined as follows. $T(\psi, \rho)$ first queries all the $\le N^{\frac{11}{12}\epsilon} < N^\epsilon/{2}$ variables from blocks $\tuple{b, y_1, y_2, \cdot}$ with $M < N^{\epsilon/12}$ that are starred in $\rho$. 
Let $\tau$ be the assignment to those variables given by a specific branch of the tree. If $\tau$ is inconsistent with $g(\rho)$, the branch ends immediately in a leaf labelled (for example) $0$. Otherwise, the subtree of $T(\psi, \rho)$ below $\tau$, say $T^\tau(\psi, \rho)$, is defined by the following process. Let $C_1$ be the first conjunction in $\psi$ such that $C_1|_{\rho\tau} \not \equiv 0$. Let $\beta_1$ be the list of blocks which contain a starred variable appearing in $C_1|_{\rho\tau}$. For each such block 
$\tuple{b, y_1, y_2, \cdot}$, the tree queries $\alpha_{b, y_1, y_2, y_3}$ for the unique $y_3$ for which this is starred in $g(\rho)$ (even if this is not the variable appearing in  $C_1|_{\rho\tau}$). Let $\pi_1$ be the assignment to the blocks in $\beta_1$ given by $g(\rho)$ together with the answers to the queries made by the tree. If $C_1|_{\rho\tau\pi_1} \equiv 1$, the branch ends in a leaf labelled $1$. If not, let $C_2$ be the first conjunction in $\psi$ such that $C_1|_{\rho\tau} \not \equiv 0$, let $\beta_2$ consist of all blocks containing a starred variable appearing in $C_2|_{\rho\tau\pi_1}$ and continue as previously. If at any point
$C_i|_{\rho\tau\pi_1\ldots\pi_i} \equiv 1$, the branch ends in a leaf labelled $1$. Otherwise, at some point all cojunctions in $\psi$ have been falsified and the tree ends in a leaf labelled $0$. By construction, $T(\psi, \rho)$ correctly determines the truth value of $\psi|_{g(\rho)}$.

Let $S$ be the set of the $\rho$ such that there exists $\tau$ for which $T^\tau(\psi, \rho)$ has height at least 
$N^\epsilon/2$. To prove  the lemma, it suffices to show
\[\Pr_\rho[\rho \in S ] \le \frac{1}{2^{N^\epsilon}}.\]

Let $\rho$ be in $S$, let $\tau$ be the first assignment such that $T^\tau(\psi, \rho)$ has height at least $N^\epsilon/2$, and let $\pi$ be the first branch in $T^\tau(\psi, \rho)$ of length at least $N^\epsilon/2$. Let $C_1, \ldots, C_k$, $\beta_1, \ldots, \beta_k$, $\pi_1,\ldots,\pi_k$ be the conjunctions, sets of  blocks, and assignments encountered in the first $N^\epsilon/2$ queries in $\pi$. 
For $i =1, \ldots, k$, let $\gamma_i$ be the set of the starred variables from blocks in $\beta_i$ appearing positively in $C_i$.

For $i = 1, \ldots, k$, the assignment $\sigma_i$ is defined as follows. $\sigma_1$ sets each variable appearing in $\gamma_i$ to $1$ and sets all other starred variables from blocks in $\beta_i$ to $0$. Note that $\sigma_i$ sets exactly the same variables as $\pi_i$, so that the $\sigma_i$'s have pairwise disjoint domains and thus $\sigma = \sigma_1\ldots\sigma_k$ is a well-defined restriction. Note also that $C_i|_{\rho\tau\pi_1\ldots\pi_{i-1}\sigma_i} \not\equiv 0$ (in fact, $C_i|_{\rho\tau\pi_1\ldots\pi_{i-1}\sigma_i} \equiv 1$ for all $i  = 1, \ldots, k-1$).

We can represent $\tau$ by a number below $3^{N^\epsilon/2}$ (each of $N^\epsilon/2$ variables can be set to $1$, $0$ or already evaluated in $\rho$). We code $\tuple{\beta_1, \ldots, \beta_k}$ as a tuple $\beta' = \tuple{\beta'_1, \ldots, \beta'_k}$, where for each block in $\beta_i$, $\beta'_i$ contains the information which of the $N^\delta$ variables in $C_i$ belong to the block and an additional bit to determine whether the next block listed in $\beta'$ is still in $\beta_i$ or already in $\beta_{i+1}$. Since there are at most $N^\epsilon/2$ blocks in all the $\beta_i$'s taken together, $\beta'$ can be represented by a number below $(2N^\delta)^{N^\epsilon/2}$. We code $\tuple{\pi_1, \ldots, \pi_k}$ as a number $\pi'$ below $2^{N^\epsilon/2}$ (a string of $N^\epsilon/2$ bits listing the answers to queries given along $\pi$). Finally, we code $\tuple{\gamma_1, \ldots, \gamma_k}$ as a tuple 
$\gamma' = \tuple{\gamma'_1, \ldots, \gamma'_k}$, where each $\gamma'_i$ contains the information which of the $N^\delta$ variables in $C_i$ comprise $\gamma_i$. Since there are at most $N^\epsilon/2$ $C_i$'s, $\gamma'$ can be represented by a number below $2^{N^{\delta+\epsilon}/2}$.

We claim that the mapping:
\[S \ni \rho \mapsto \tuple{\rho\tau\sigma, \tau, \beta', \pi', \gamma'} \]
is an injection. To see this, note that given $\tuple{\rho\tau\sigma, \tau, \beta', \pi', \gamma'}$, we can find
$C_1$ as the first conjunction $C$ in $\psi$ such that $C|_{\rho\tau\sigma}\neq 0$. Knowing $C_1$, we can decode $\beta'_1$ and $\gamma'_1$ to find $\beta_1, \gamma_1$. We can now identify $\sigma_1$ as the assignment which sets all variables in $\gamma_1$ to $1$ and sets all the variables from blocks in $\beta_1$ which are set to $0$ by $\rho\tau\sigma$. Knowing $\sigma_1$, $\gamma_1$ and $\pi'_1$, we can determine $\pi_1$, identify $C_2$ as the first conjunction $C$ in $\psi$ such that $C|_{\rho\tau\pi_1\sigma_2\ldots\sigma_k}\neq 0$, identify $\beta_2, \gamma_2, \sigma_2, \pi_2$, and so on. Eventually, we are able to recover all of $\sigma$ and therefore also $\rho\tau$. Since we are also given $\tau$, this is enough to recover $\rho$.

For fixed values of $\tau, \beta', \pi', \gamma'$, let $S_{\tau, \beta', \pi', \gamma'}$ consist of those $\rho \in S$ to which these specific values are assigned. By the previous paragraph, the mapping $\rho \mapsto \rho\tau\sigma$ restricted to $S_{\tau, \beta', \pi', \gamma}$ is an injection. For $\rho \in S_{\tau, \beta', \pi', \gamma}$, the restriction $\rho \tau$ is obtained from $\rho$ by turning some $*$'s into $1$'s and some $*$-blocks into $0$-blocks, so certainly the probability of $\rho\tau$ is no smaller than that of $\rho$. Going from $\rho\tau$ to $\rho\tau \sigma$ changes $N^\epsilon/2$ blocks which were previously $*$-blocks into $0$-blocks, and changes some $*$'s into $1$'s. Since all variables from blocks with $M  < N^{\epsilon/12}$ are assigned $0/1$ values in $\rho\tau$, each change of a $*$-block into a $0$-block or of a $*$ into $1$ increases the probability of a restriction by a factor of at least 
\[\left.\left(1 - \frac{1}{N^{\epsilon/12}}\right)\middle/\left( \frac{1}{N^{\epsilon/12}} \right) \right. = N^{\epsilon/12} - 1.\] 
If we write $\ell$ to denote of $*$'s changed into $1$'s, then $\ell$ equals the number of positively appearing variables listed in $\gamma'$, so $0 \le \ell \le  N^{\delta + \epsilon}/2$. The ratio $\Pr(\rho \tau \sigma)/\Pr(\rho)$ is at least
\[\left(N^{\epsilon/12} - 1\right)^{{N^\epsilon}/{2} + \ell}\]
and the probabilities of the $\rho\tau\sigma$'s cannot add up to more than 1, so
\[\Pr_\rho[\rho \in S_{\tau, \beta', \pi', \gamma}] \le \left(N^{\epsilon/12} - 1\right)^{-{N^\epsilon}/{2} - \ell}.\]

For fixed $\tau, \beta', \pi'$, write $S_{\tau, \beta', \pi'}$ for $\bigcup_{\gamma'} S_{\tau, \beta', \pi', \gamma}$. We have
\begin{align*}
\Pr_\rho [\rho \in S_{\tau, \beta', \pi'}] = \sum_{\gamma'}\Pr_\rho [\rho \in S_{\tau, \beta', \pi', \gamma'}] = \\
= \sum_{\ell = 0}^{N^{\delta+\epsilon}/2}\sum_{|\gamma'|=\ell}\Pr_\rho [\rho \in S_{\tau, \beta', \pi', \gamma'}] \le \\
 \le \sum_{\ell = 0}^{N^{\delta+\epsilon}/2}\binom{N^{\delta+\epsilon}/2}{\ell}\left(N^{\epsilon/12} - 1\right)^{-{N^\epsilon}/{2} - \ell} = \\ 
= \left(N^{\epsilon/12} - 1\right)^{-{N^\epsilon}/{2}}\left(\left. N^{\epsilon/12} \middle/ \left( N^{\epsilon/12} - 1\right) \right.\right)^{N^{\delta+\epsilon}/2}\le \\
\le  \left(N^{\epsilon/12} - 1\right)^{-{N^\epsilon}/{2}} e^{N^\epsilon/2}.
\end{align*}
The last inequality holds for sufficiently large $N$ by the assumption that $12\delta< \epsilon$. Now, $S = \bigcup_{{\tau, \beta', \pi'}}S_{\tau, \beta', \pi'}$, so by the union bound
\[\Pr_\rho[\rho \in S] \le \left(12eN^\delta\right)^{N^\epsilon/2} \left(N^{\epsilon/12} - 1\right)^{-{N^\epsilon}/{2}},\]
which is smaller than $2^{-N^\epsilon}$ for large enough $N$.
\end{proof}

\section{Concluding remarks}

We have formulated a complexity-theoretic statement that might be very hard to disprove and would imply that $\bsi$ is provable from $\neg \Exp$.  It would be interesting to give evidence for the hardness of disproving the statement (\ref{eqn:assumption}), for example by constructing a relativized world in which (\ref{eqn:assumption}) holds; this would also lead to a proof of our conjecture from Section 4. On the other hand, disproving (\ref{eqn:assumption}) would be no less interesting, and could lead to some progress towards a proof that $\ido + \Exp \not\vdash \bsi$.

The question concerning provability of $\bsi$ is just one of many classical questions about weak arithmetic which do not explicitly mention complexity theory but seem to have it lurking in the background. One other example is the problem whether every model of $\ido + \bsi$ has a proper end-extension to a model of $\ido$. It could be interesting to find, on the one hand, ``plausible'' complexity-theoretic statements implying the ``expected'' answers to these questions, and on the other hand, ``plausibly hard to disprove'' statements implying the ``unexpected'' answers. The former task, in fact, not yet fully carried out even in the case of the question studied in this paper. The best known complexity-theoretic statement implying $\ido + \Exp \not\vdash \bsi$, namely $\ph{\downarrow}$, is usually conjectured to be false, while other such statements are more obscure and seem hard to make any conjecture about.

\end{document}